\documentclass[reqno,a4paper,12pt]{amsart}
\usepackage[top=32mm, bottom=32mm, left=32mm, right=32mm]{geometry}
\usepackage{mathptmx}
\usepackage{amssymb,latexsym}
\usepackage{enumerate}
\usepackage{color}
\usepackage[mathcal]{euscript}
\usepackage{mathrsfs}
\usepackage{graphicx}
\usepackage[colorlinks=true,citecolor=blue]{hyperref}

\newtheorem{thm}{Theorem}[section]
\newtheorem{thmx}{Theorem}

\newtheorem{lem}[thm]{Lemma}
\newtheorem{prop}[thm]{Proposition}
\newtheorem{cor}[thm]{Corollary}
\newtheorem*{claim}{Claim}
\newtheorem{que}{Question}
\theoremstyle{definition}
\newtheorem{defn}[thm]{Definition}
\newtheorem{rem}[thm]{Remark}

\numberwithin{equation}{section}

\newcommand{\N}{\mathbb{N}}
\newcommand{\Z}{\mathbb{Z}}
\newcommand{\R}{\mathbb{R}}

\newcommand{\orbp}{\mbox{Orb}}
\newcommand{\set}[1]{\left\{#1\right\}}
\newcommand{\eps}{\varepsilon}

\DeclareMathOperator{\Trans}{Trans}

\title{Invariant scrambled sets, uniform rigidity and weak mixing}
\author[M. Fory\'s]{Magdalena Fory\'s}
\address[M. Fory\'s]{AGH University of Science and Technology\\
Faculty of Applied Mathematics\\
al. A. Mickiewicza 30, 30-059 Krak\'ow,
Poland\\ -- and --\\Institute of Computer Science, Faculty of Mathematics and Computer Science, Jagiellonian University,
ul. \L ojasiewicza 6, 30-348 Krak\'ow, Poland}\email{magdalena.forys@uj.edu.pl}

\author{Wen Huang}
\address[W.~Huang]{Department of Mathematics, Sichuan University,
Chengdu, Sichuan 610064, China \\ -- and --\\
School of Mathematical Sciences, University of Science and Technology of
China, Hefei, Anhui 230026, China}
\email{wenh@mail.ustc.edu.cn}

\author[J. Li]{Jian Li}
\address[J. Li]{Department of Mathematics, Shantou University, Shantou, Guangdong, 515063, P.R. China}
\email{lijian09@mail.ustc.edu.cn}
\thanks{Corresponding author: Jian Li (lijian09@mail.ustc.edu.cn)}
\author[P. Oprocha]{Piotr Oprocha}
\address[P. Oprocha]{AGH University of Science and Technology\\
Faculty of Applied Mathematics\\
al. A. Mickiewicza 30, 30-059 Krak\'ow,
Poland\\ -- and --\\Institute of Mathematics\\ Polish Academy of Sciences\\ ul. \'Sniadeckich 8, 00-956 Warszawa, Poland} \email{oprocha@agh.edu.pl}
\date{\today}
\begin{document}
\begin{abstract}
We show that for a non-trivial transitive dynamical system,
it has a dense Mycielski invariant strongly scrambled set if and only if it has a fixed point,
and it has a dense Mycielski invariant $\delta$-scrambled set for some $\delta>0$
if and only if it has a fixed point and not uniformly rigid.
We also provide two methods for the construction of completely scrambled systems
which are weakly mixing, proximal and uniformly rigid.
\end{abstract}
\maketitle

\section{Introduction}
The first mathematical treatment of chaotic behavior of a dynamical system appeared in
the work of Li and Yorke in 1975~\cite{LY75}.
A \emph{(topological) dynamical system} is a pair $(X,f)$, where $X$ is a compact metric space with a metric $d$
and $f:X\to X $ is a continuous map.
A subset $S$ of $X$ with at least two points is \emph{scrambled} if for any $x,y\in S$ with $x\neq y$, one has
\begin{align}
\liminf_{n\to\infty}d(f^n(x),f^n(y))&=0
\intertext{and}
\limsup_{n\to\infty} d(f^n(x),f^n(y))&>0.\label{def:scr:c2}
\end{align}
If $X$ contains an uncountable scrambled subset, then
the dynamical system $(X,f)$ is called \emph{chaotic in the sense of Li and Yorke}.
A lot of attention has been paid to the study of scrambled sets.
We refer the reader to~\cite{BHS} for a comprehensive treatment of topological size of scrambled sets.

It is still an open problem whether a Cantor or Mycielski scrambled set can be selected when an
uncountable scrambled set exists in a dynamical system.
But we know that the answer is positive for the following two strengthened definitions of scrambled sets.
If we replace condition \eqref{def:scr:c2} in the definition of scrambled set by the demand that
$(x,y)$ is recurrent in $(X\times X,f\times f)$,
then we obtain the notion of \emph{strongly scrambled set}, and replacing \eqref{def:scr:c2} by
$\limsup_{n\to\infty} d(f^n(x),f^n(y))\geq \delta$
(with some $\delta>0$ depending only on $S$) we obtain the notion of \emph{$\delta$-scrambled set.}
If a dynamical system has an uncountable strongly scrambled set
(resp. an uncountable $\delta$-scrambled set for some $\delta>0$,
then it has a Mycielski strongly scrambled set~\cite{A04} (resp. a Mycielski $\delta$-scrambled set~\cite{BHS}).

Even in the case of compact interval, it is known
that there are some continuous maps having scrambled sets  and at the same time with
zero topological entropy.
It was first observed in~\cite{Du05} that an interval map has positive topological
entropy if and only if some of its iterates possesses an invariant scrambled set.

It is an interesting question, under which conditions a scrambled set can be chosen to be invariant.
The strongest possible situation takes place, when the whole space $X$ can be a scrambled set,
and we call $(X,f)$  \emph{completely scrambled} if this condition holds.
An immediate consequence of~\eqref{def:scr:c2} is that such a dynamical system must be injective and
$f$ is a homeomorphism provided that it is surjective.
In~\cite{HYCS}, Huang and Ye constructed examples of compacta with completely scrambled homeomorphisms.
This allows them to show that for every integer $n\geq 1$
there is a dynamical system on a continuum of topological dimension $n$  which is completely scrambled.
Their examples however are not transitive. They also mention in \cite{HYCS} that an example of completely scrambled
transitive homeomorphism is a consequence of construction of uniformly rigid proximal systems by Katznelson and Weiss~\cite{KW}.
Later in~\cite{HY02}, Huang and Ye showed that
every almost equicontinuous but not minimal system has a completely scrambled factor.
These examples are not weakly mixing, however,
so existence of completely scrambled weakly mixing homeomorphism is left open in~\cite{HYCS}.
In this paper, we will provide two methods for the construction of completely scrambled systems
which are weakly mixing, proximal and uniformly rigid.
The first possible approach is  derived from results of Akin and Glasner~\cite{AG} by
a combination of abstract arguments.
The second method is obtained by modifying the construction of Katznelson and Weiss from~\cite{KW}.
In particular we have the following theorem.

\begin{thmx}\label{thm:A}
There are completely scrambled systems
which are weakly mixing, proximal and uniformly rigid.
\end{thmx}

Note that the uniformly rigid property was introduced by Glasner and Maon in~\cite{GM89} as a topological
analogue of rigidity in ergodic theory.
Recall that a dynamical system $(X,f)$ is \emph{uniformly rigid} if
\[\liminf_{n\to\infty}\sup_{x\in X} d(f^n(x),x)=0.\]
They constructed a class of minimal systems
which are both uniformly rigid and weakly mixing in~\cite{GM89}.
But those systems cannot be completely scrambled.
The dynamical systems in Theorem~\ref{thm:A} are also the first examples on
non-minimal weakly mixing systems with the uniformly rigid property.

Even though we do not know whether a Cantor or Mycielski invariant scrambled set can be selected when an
uncountable invariant scrambled set exists in a dynamical system, we show that
uncountable invariant scrambled sets can be dense in a suitable subsystem.
We also show that the answer is positive for invariant strongly scrambled sets and  invariant $\delta$-scrambled sets.

In 2002, Huang and Ye proved that a non-periodic transitive system with a periodic point contains
an uncountable scrambled sets~\cite{HY02}.
In fact by their proof the scrambled set can be chosen to be a Mycielski strongly scrambled set.
Moreover, if this system also has the sensitive dependence
on initial conditions, then it contains a Mycielski $\delta$-scrambled set for some $\delta>0$~(see~\cite{M14}).

It is natural to ask when a transitive system contains invariant strongly scrambled sets or
invariant $\delta$-scrambled sets.
First, it is easy to see that if  a dynamical system contains an invariant scrambled set, then it must contain
a fixed point. We show that this condition is also sufficient for transitive systems.

\begin{thmx}\label{thm:B}
If a non-trivial dynamical system $(X,f)$ is transitive,
then $(X,f)$ contains a dense Mycielski invariant strongly scrambled set
if and only if it has a fixed point.
\end{thmx}

Note that $\delta$-scrambled sets are never compact and invariant at the same time~\cite{BHS}.
It is shown in~\cite{BGO10} that every non-trivial topologically mixing map with at least one fixed point contains invariant $\delta$-scrambled sets.
They also conjecture in \cite{BGO10} that there exists a weakly mixing system with a fixed point
but without invariant $\delta$-scrambled sets.

We find that the existence of invariant $\delta$-scrambled sets is related to the uniformly rigid property.
We show that a necessary condition for a dynamical system
possessing invariant $\delta$-scrambled sets for some $\delta>0$  is not uniformly rigid,
and this condition (with a fixed point) is also sufficient for transitive dynamical systems.

\begin{thmx}\label{thm:C}
If a non-trivial dynamical system $(X,f)$ is transitive, then $(X,f)$ contains
a dense Mycielski invariant $\delta$-scrambled set for some $\delta>0$
if and only if it has a fixed point and is not uniformly rigid.
\end{thmx}

As a byproduct of Theorem~\ref{thm:C}, the above mentioned conjecture from~\cite{BGO10} has an affirmative answer,
because in Theorem~\ref{thm:A} we construct a weakly mixing, proximal and uniformly rigid system with a fixed point but without invariant $\delta$-scrambled set.

The notion of mild mixing was first introduced in ergodic theory by Furstenberg and Weiss in~\cite{FG78}.
An ergodic system is \emph{mildly mixing} if it has no non-trivial rigid factor.
It is shown in~\cite{FG78} that mildly mixing systems are disjoint from rigid systems, and
mild mixing property is equivalent to the following multiplier property. Roughly speaking,
an ergodic system is mildly mixing if and only if for every
ergodic (finite or infinite) measure preserving system, the product of these two system is ergodic.
The topological
analogue of mild mixing in ergodic theory was introduced by Glasner and Weiss~\cite{GW06}
and independently by Huang and Ye~\cite{HY04}.
A dynamical system $(X,f)$ is \emph{mildly mixing} if for every transitive system $(Y,g)$
the product system $(X\times Y,f\times g)$ is transitive.
It is clear that every strongly mixing system is mildly mixing, but
there exists a mildly mixing system which is not strongly mixing~\cite{HY04}.
In~\cite{GM89} the existence of minimal weakly mixing and uniformly rigid
dynamical systems is demonstrated.
However, a dynamical system which is both mildly mixing and uniformly rigid
must be trivial~\cite{GW06,HY04}.
Then a directly corollary of Theorem~\ref{thm:C} is that
each non-trivial mildly mixing system has a dense Mycielski invariant $\delta$-scrambled set
if and only if it has a fixed point.

\section{Preliminaries}

Denote by $\N$ the set of all positive integers and we denote the set of all non-negative integers by $\N_0=\N\cup \set{0}$.
A subset $A$ of $\N$ is \emph{syndetic} if there is $k>0$ such that $[i,i+k]\cap A\neq \emptyset$ for every $i\in \N$.

\subsection{Subsets in compact metric spaces}
Let $(X,d)$ be a compact metric space.
The \emph{diagonal} of $X\times X$ is denoted by
$\Delta_X=\set{(x,x) : x\in X}$ (when the set $X$ is clear from the context we simply write $\Delta$).
A subset $A$ of $X$ is \emph{non-trivial} if it consists of at least two points.
If a subset $A$ of $X$ can be presented as intersection of
countable many open sets then we call it $G_\delta$ and $A$ is \emph{residual} if it is $G_\delta$ and dense in $X$.
A non-empty subset of $X$ is: \emph{perfect} if it is closed and has no isolated points
(each open subset contains at least two points);
\emph{Cantor} if it is perfect and zero dimensional (has a base of the topology consisting of sets which are both closed and open);
\emph{Mycielski} if it can be presented as a countable union of Cantor sets.
For convenience, we state here a simplified version of Mycielski theorem~\cite[Theorem~1]{M64} which we shall use.

\begin{thm}[Mycielski Theorem]
Let $X$ be a perfect compact metric space.
If $R$ is a residual subset of $X\times X$,
then there exists a dense Mycielski set $M\subset X$ such that $M\times M\subset R\cup\Delta$.
\end{thm}

Let $A$ be an uncountable subset of $X$.
A point $x\in X$ is a \emph{condensation point} of $A$
if for every neighborhood $U$ of $x$, $U\cap A$ is uncountable.
Denote the set of all condensation points of $A$ by $A^*$. It is known (e.g. see \cite[\S23]{KurVolI}) that $A^*$ is closed, $(A\cap A^*)^*=A^*$ and
$A\setminus A^*$ is at most countable.
For every $p\in A^*$ and every neighborhood $U$ of $p$, the set $A\cap U$ is uncountable
and $A\setminus A^*$ is countable, hence $A^* \cap U$ is also uncountable, in particular $A^*$ is perfect.

\subsection{Topological dynamics}
A \emph{(topological) dynamical system} is a pair $(X,f)$ consisting of a compact metric space $(X,d)$
and a continuous map $f\colon X\to X$.
We say that $x\in X$ is a \emph{fixed point of $(X, f)$} if $f (x)= x$;
a \emph{periodic point of $(X, f)$} if $f^n (x)= x$ for some $n\in \N$; a \emph{recurrent point of $(X, f)$} if
$\liminf_{n\to\infty}d(f^n(x),x)=0$.
We denote the \emph{(positive) orbit} of $x$ by $\orbp(x,f)=\set{x,f(x),f^2(x),\ldots}$.
The \emph{$\omega$-limit set of $x$} is defined by $\omega(x,f)=\bigcap_{n=1}^\infty \overline{\orbp(f^n(x),f)}$.
Note that a point $x\in X$ is recurrent if and only if $x\in \omega(x,f)$.
A dynamical system $(X,f)$ is \emph{pointwise recurrent} if every $x\in X$ is recurrent.

A subset $D \subset X$ is \emph{$f$-invariant} (or simply \emph{invariant}) if $f(D)\subset D$.
A non-empty closed invariant subset $D$ of $X$ is \emph{minimal}, if $\overline{\orbp(x,f)} = D$ for every $x \in D$.
A point $x \in  X$ is \emph{minimal} if it is contained in a minimal subset of $X$.

For $x\in X$ and $U,V\subset X$, we write $N(x,U)=\set{n\in\N: f^n(x)\in U}$ and
$N(U,V)=\set{n\in \N: f^n(U)\cap V\neq\emptyset}$.
A dynamical system $(X,f)$ is \emph{transitive} if $N(U,V)\neq\emptyset$
for any non-empty open sets $U,V\subset X$; \emph{weakly mixing} if $(X\times X, f\times f)$ is transitive;
\emph{strongly mixing} if $N(U,V)$ is cofinite  for any non-empty open sets $U,V\subset X$.
A point $x\in X$ is a \emph{transitive point} if its orbit $\orbp(x,f)$ is dense in $X$.
It is well known that for a transitive system $(X,f)$, the collection of all transitive points forms
a dense $G_\delta$ subset of $X$.

Next lemma provides an accessible tool for a construction of weakly mixing systems
(see \cite[Lemma~5.1]{HY05} or \cite[Lemma~3.8]{OprWM}).

\begin{lem}\label{lem:WM}
Assume that $(X,f)$ is transitive and that $x\in X$ has a dense orbit. Then $(X,f)$ is weakly mixing if and only if
for any open neighborhood $U$ of $x$ there is $n>0$ such that $n,n+1\in N(U,U)$.
\end{lem}

We will use the following two results which were proved in~\cite[Lemma 2.2 and Theorem 4.3]{HYBeb}.

\begin{lem}\label{lem:proximal}
Let $(X,f)$ be a transitive system and $x\in X$ be a transitive point.
If $A$ is a closed invariant subset of $ X$ such that $N(x,U)$ is syndetic for each neighborhood of $A$, then each minimial set of $(X,f)$ is contained in $A$.
\end{lem}

\begin{lem}\label{lem:connected}
Let $(X,f)$ be transitive and pointwise recurrent.
If $(X,f)$ contains a minimial set that is connected, then $X$ is connected.
\end{lem}

Let $(X,f)$ and $(Y,g)$ be two dynamical systems.
If there is a continuous surjection $\pi\colon X\to Y$ which intertwines the actions (i.e., $\pi\circ f=g\circ \pi$),
then we say that $\pi$ is a \emph{factor map},
$(Y,g)$ is a \emph{factor} of $(X,f)$ or $(X,f)$ is an \emph{extension} of $(Y,g)$.

Let $(X,f)$ be a dynamical system.
If there exists an increasing sequence $\{n_k\}_{k=1}^\infty$ in $\N$ such that
$$
\lim_{k\to \infty}\sup_{x\in X}d(f^{n_k}(x),x)=0
$$
then we say that $(X,f)$ is \emph{uniformly rigid}.
In other words $\set{f^{n_k}}_{k=1}^\infty$ converges uniformly to the identity map on $X$.
Clearly, each uniformly rigid system is pointwise recurrent.

A point $x\in X$ is an \emph{equicontinuity point} if for every $\eps>0$ there is $\delta>0$
such that if $d(x,y)<\delta$ then $d(f^n(x),f^n(y))<\eps$ for every $n\geq 0$.
A transitive system $(X,f)$ is \emph{almost equicontinuous}
if every transitive point is equicontinuous.
It is known (see~\cite[Lemma~1.2]{GW93} or \cite[Corollary~3.7]{AAB})
that every almost equicontinuous system is uniformly rigid.
It is shown in~\cite[Propsition~1.5]{GW93} that every uniformly rigid transitive system
has an almost equicontinuous extension.
Since the uniformly rigid property is preserved under factor maps,
a transitive system is uniformly rigid if and only if it has an almost equicontinuous extension.

\section{Proximality and scrambled sets}

Let $(X,f)$ be a dynamical system.
A pair of points $(x,y)\in X\times X$ is \emph{proximal} if $\liminf_{n\to\infty} d(f^n(x),f^n(y))=0$, and
\emph{asymptotic} if $\lim_{n\to\infty} d(f^n(x),f^n(y))=0$.
It is convenient to study these concepts using various relations on $X$, that is, subsets of $X \times X$.
For any $\varepsilon>0$, denote by
\begin{align*}
  V_\varepsilon=\{(x,y)\in X\times X\colon d(x,y)<\varepsilon\},\quad
  \overline{V}\!_\varepsilon=\{(x,y)\in X\times X\colon d(x,y)\leq \varepsilon\}.
\end{align*}
For any subset $R\subset X \times X$ and any point $x\in X$, we write
\[R(x)=\{y\in X\colon (x,y)\in R\}.\]
So, for example,  $V_\varepsilon$ (or $\overline{V}\!_\varepsilon$)
is the open (respectively closed) ball of radius $\varepsilon$ centred at $x$.
We use these to define the sets of proximal pairs and asymptotic pairs:
\begin{align*}
  \mathrm{Prox}(f)&=\bigl\{(x,y)\in X\times X\colon \liminf_{n\to\infty} d(f^n(x),f^n(y))=0\bigr\}\\
                  &=\bigcap_{\varepsilon>0}\bigcap_{N\geq 1} \bigcup_{n\geq N} (f\times f)^{-n} (V_{\varepsilon}),\\
  \mathrm{Asym}(f)&=\bigl\{(x,y)\in X\times X\colon \lim_{n\to\infty} d(f^n(x),f^n(y))=0\bigr\}\\
                  &=\bigcap_{\varepsilon>0}\bigcup_{N\geq 1} \bigcap_{n\geq N} (f\times f)^{-n} (\overline{V}\!_{\varepsilon}).
\end{align*}
We start with the following remark.
\begin{rem}\label{rem:invariant-s-s}
If $(x,f(x))\in \mathrm{Prox}(f)$, then by compactness of $X$ there exists an increasing sequence
$\{n_i\}$ in $\N$ and a point $p\in X$
such that $\lim_{i\to\infty}d(f^{n_i}(x),f^{n_i}(f(x)))=0$ and $\lim_{i\to\infty}f^{n_i}(x)=p$,
which implies that $f(p)=p$ by continuity of $f$, that is $p$ is a fixed point.
In this case $(f^n(x),f^m(x))\in  \mathrm{Prox}(f)$ for all $m,n\geq 0$.
\end{rem}

We say that a dynamical system $(X,f)$ is \emph{proximal} if each pair in $X\times X$ is proximal,
i.e. $\mathrm{Prox}(f)=X\times X$.
We have the following characterization of proximal systems, which is essential contained in
 \cite[Proposition 2.2]{AK} and \cite[Proposition 2.2]{HYCS}.
Here we provide a proof for the completeness.

\begin{prop}
Let $(X,f)$ be a dynamical system. The following conditions are equivalent:
\begin{enumerate}
  \item $(X,f)$ is proximal;
  \item $(x,f(x))\in \mathrm{Prox}(f)$ for all $x\in X$ and $X$ contains a unique fixed point;
  \item $X$ contains a fixed point which is the unique minimal point of $X$.
\end{enumerate}
\end{prop}
\begin{proof}
(1)$\Rightarrow$(2) follows from the definition and Remark \ref{rem:invariant-s-s}.

(2)$\Rightarrow$(3) Let $p$ be the unique fixed point and $q$ be a minimal point in $X$.
By $(q,f(q))\in \mathrm{Prox}(f)$, there exists a fixed point in $\overline{\orbp(q,f)}$.
Then $p\in \overline{\orbp(q,f)}$. But $q$ is a minimal point, so $q=p$,
that is $p$ the unique minimal point of $X$.

(3)$\Rightarrow$(1)
If a fixed point $p$ is the unique minimal point of $X$ then $(p,p)$
is the unique minimal  point of $X \times X$. Hence, for every pair $(x, y) \in  X \times X$ one has
$(p, p) \in \omega((x, y),f\times f)$ and so $(x, y)\in \mathrm{Prox}(f)$.
\end{proof}

We say that a pair of points $(x,y)\in X\times X$ is \emph{scrambled} if it is proximal but not asymptotic,
and a subset $S$ of $X$ containing at least two points
is \emph{scrambled} if every pair of distinct points of $S$ is scrambled.
The set of scrambled pairs is denoted by $\mathrm{LY}(f)=\mathrm{Prox}(f)\setminus \mathrm{Asym}(f)$.
It is clear that $\mathrm{Prox}(f)$ is $G_\delta$ and $\mathrm{Asym}(f)$ is $F_{\sigma\delta}$.
But in general $\mathrm{LY}(f)$ is not a $G_\delta$ relation on $X$,
thus in various cases it is not possible to apply the Mycielski Theorem to get scrambled sets.
To overcome this disadvantage, we define two stronger version of scrambled pairs.

For a given positive number $\delta>0$, we say that a pair is \emph{$\delta$-scrambled}
if $(x,y)$ is proximal and $\limsup_{n\to\infty}d(f^n(x),f^n(y))\geq \delta$.
Denote by
\begin{align*}
  \mathrm{Asym}_\delta(f)&=\bigl\{(x,y)\in X\times X\colon \limsup_{n\to\infty} d(f^n(x),f^n(y))<\delta\bigr\}\\
                  &=\bigcup_{\varepsilon<\delta}\bigcup_{N\geq 1}
                  \bigcap_{n\geq N} (f\times f)^{-n} (\overline{V}\!_{\varepsilon}).
\end{align*}
The set of $\delta$-scrambled pairs, denote by $\mathrm{LY}_\delta(f)$,
is  $\mathrm{Prox}(f)\setminus \mathrm{Asym}_\delta(f)$.
It is clear that $\mathrm{Asym}_\delta(f)$ is $F_\sigma$, then $\mathrm{LY}_\delta(f)$ is a $G_\delta$ relation on $X$.

A pair $(x,y)$ is \emph{strongly scrambled}
if $(x,y)$ is proximal for $(X,f)$ and recurrent in $(X\times X,f\times f)$.
This notion was first introduced by Akin in~\cite{A04}.
We can define the set $\mathrm{Recur}(f)$ of recurrent points of $(X,f)$ equivalently by
\begin{align*}
  \mathrm{Recur}(f)&=\{x\in X\colon \liminf_{n\to\infty} d(f^n(x),x)=0\}\\
          &=\bigcap_{\varepsilon>0} \bigcap_{N\geq 1}\bigcup_{n\geq N}
          \bigcup_{x\in X}\bigl(V_{\varepsilon}(x)\cap f^{-n}(V_{\varepsilon}(x))\bigr).
\end{align*}
We denote by $\mathrm{sLY}(f)$ the set of strongly scrambled pairs, that is $\mathrm{sLY}(f)=\mathrm{Prox}(f)\cap \mathrm{Recur}(f\times f)$.
It is clear that $\mathrm{Recur}(f)$ is $G_\delta$, thus $\mathrm{sLY}(f)$ is a $G_\delta$ relation on $X$.

Similarly as before, we say that a subset $S$ of $X$ containing at least two points
is \emph{strongly scrambled} (resp. \emph{$\delta$-scrambled} for some $\delta>0$)
if every pair of distinct points of $S$ is strongly scrambled (resp. $\delta$-scrambled).
Now using the Mycielski Theorem, we give a easy proof the result that a Mycielski strongly scrambled set
(resp. a Mycielski $\delta$-scrambled set)
can be selected when an uncountable strongly scrambled set
(resp. an uncountable $\delta$-scrambled set for some $\delta>0$)
exists in a dynamical system, which was first proved in \cite[Theorem 6.10]{A04} (resp. \cite[Theorem 16]{BHS}).
\begin{thm}\label{thm:uncountable-to-Mycielski-1}
Let $(X,f)$ be a dynamical system.
\begin{enumerate}
\item If $(X,f)$ has an uncountable strongly scrambled set, then
it also has a Mycielski strongly scrambled set.
\item If $(X,f)$ has an uncountable $\delta$-scrambled set for some $\delta>0$,
then it also has a Mycielski $\delta$-scrambled set.
\end{enumerate}
\end{thm}
\begin{proof}
(1) Let $K$ be an uncountable strongly scrambled subset of $X$.
Let $K^*$ be the set of condensation points of $K$.
Then $K^*$ is closed, perfect
and $\mathrm{sLY}(f)\cap K^*\times K^*$ is a dense  $G_\delta$ subset of $K^*$,
because $K\cap K^*$ is dense in $K^*$ and $K\times K\setminus \Delta\subset \mathrm{sLY}(f)$.
By the Mycielski Theorem, there exists a dense Mycielski subset $S$ of $K^*$
such that $S\times S\setminus \Delta\subset \mathrm{sLY}(f)$, that is $S$ is strongly scrambled.

(2) The proof is similar to (1) by replacing $\mathrm{sLY}(f)$ with $\mathrm{LY}_\delta(f)$.
\end{proof}

A dynamical system $(X,f)$ is \emph{completely scrambled} if the whole space $X$ is a scrambled set.
In~\cite{HYCS}, Huang and Ye showed that there are ``many'' compacta admitting completely scrambled
homeomorphisms, which include some countable compacta,
the Cantor set and continua of arbitrary dimension.
In~\cite{BHS},  Blanchard, Huang and Snoha showed
that the whole space $X$ cannot be a $\delta$-scrambled  set for any $\delta>0$.
Note that one can also obtain this from a result of Schwartzman (e.g. see \cite[Theorem~2.1]{King}) that is,
if $X$ is infinite then for every $\eps>0$ there is $x\neq y\in X$ such that $d(f^n(x),f^n(y))<\eps$ for all $n\geq 1$.

\section{Invariant scrambled sets}
Let $(X,f)$ be a dynamical system.
A subset $S$ of $X$ is an \emph{invariant scrambled set} (resp. an \emph{invariant strongly scrambled set},
an \emph{invariant $\delta$-scrambled} with $\delta>0$)
if it is invariant and  scrambled (resp. invariant and  strongly scrambled, invariant and $\delta$-scrambled).
The main aim of this section is to show when a transitive  system has an invariant scrambled set.

We first show that uncountable invariant scrambled sets can be dense in a suitable subsystem.
\begin{thm}\label{thm:un-iss}
If a dynamical system $(X,f)$ has an uncountable invariant scrambled set,
then there exists a subsystem $(Y,f)$ of $(X,f)$ which has a dense uncountable invariant scrambled set.
\end{thm}
\begin{proof}
Let $K$ be an uncountable invariant scrambled subset of $X$.
Let $K^*$ be the set of condensation points of $K$.
Then $K^*$ is closed and perfect and $A=K\setminus K^*$ is at most countable.
Since $K$ is a scrambled set, $f^i|_K$ is injective,
and therefore $f^{-i}(A)\cap K$ and $f^i(A)$ are at most countable for $i\in \mathbb{N}_0$.
Let $M=K\setminus \bigcup_{i=-\infty}^\infty f^{-i}(A)$. Clearly $M$ is $f$-invariant and uncountable,
$M^*\subset K^*$ and if we take any $p\in K^*$ and any neighborhood $U$ of $p$
then $M\cap U=(K\cap U) \setminus \bigcup_{i=-\infty}^\infty f^{-i}(A)$ is uncountable and hence $p\in M^*$,
which shows that $K^*=M^*$.
Note that $M\subset \overline{M}\subset \overline{K^*}=K^*=M^*$.
Since $M$ is invariant, $K^*$ is also invariant, which in other words means $(K^*,f)$ is a subsystem of $(X,f)$.
By the construction, $M$ is dense uncountable invariant scrambled subset in $K^*$,
which ends the proof.
\end{proof}

\begin{lem}\label{lem:dense-inv-scrambled-set}
Let $(X,f)$ be a dynamical system acting on a perfect set $X$.
\begin{enumerate}
\item If $(X,f)$ has a dense invariant strongly scrambled set,
then it has a dense Mycielski invariant strongly scrambled set.\label{lem:dense-inv-scrambled-set:1}
\item If it has a dense invariant $\delta$-scrambled set for some $\delta>0$,
then it has a dense Mycielski invariant $\delta$-scrambled set.\label{lem:dense-inv-scrambled-set:2}
\end{enumerate}
\end{lem}
\begin{proof}
To prove \eqref{lem:dense-inv-scrambled-set:1}, denote by
\begin{align*}
  R&=\{(x,y)\in X\times X\colon (f^m(x),f^n(y)\in\mathrm{Recur}(f\times f) \text{ for all } m,n\geq 0\}\\
  T&=\{(x,y)\in X\times X\colon (f^m(x),f^n(x)\in\mathrm{Recur}(f\times f) \text{ for all } m,n\geq 0\}\\
  P&=\{(x,y)\in X\times X\colon (f^m(x),f^n(y)\in\mathrm{Prox}(f) \text{ for all } m,n\geq 0\}\\
\intertext{and}
  Q&=\{(x,y)\in X\times X\colon (f^m(x),f^n(x)\in\mathrm{Prox}(f) \text{ for all } m,n\geq 0\}.
\end{align*}
Since $\mathrm{Recur}(f\times f)$ and $\mathrm{Prox}(f)$ are $G_\delta$ subsets of $X\times X$,
it is easy to verify that $R$, $T$, $P$ and $Q$ are also $G_\delta$ subsets of $X\times X$.
Let $K$ be a dense invariant strongly scrambled subset of $X$.
Then
\[K\times K\subset (R\cap T\cap P\cap Q)\cup \Delta,\]
which implies that
$(R\cap T\cap P\cap Q)$ is a dense $G_\delta$ subset of $X\times X$.
By the Mycielski Theorem, there is a dense Mycielski subset $M$ of $X$ such that
$M \times M\subset (R\cap T\cap P\cap Q)\cup \Delta$.
Let $S=\bigcup_{i=0}^\infty f^i(M)$.
For two distinct points $u,v\in S$, there exist $x,y\in M$ and $m,n\in\N_0$ such that
$f^m(x)=u$ and $f^n(y)=v$.
If $x\neq y$, then $(u,v)$ is strongly scrambled since $(x,y)\in R\cap P$.
If $x=y$, we can find $z\in M$ such that $(x,z)\in R\cap T\cap P\cap Q$ and
then $(u,v)$ is strongly scrambled since $(x,z)\in T\cap Q$.
It proves that $S$ is an invariant  strongly scrambled set.

Note that $M\subset S$, hence $S$ is dense. But $S$ is does not contain asymptotic pairs, hence
$f|_S$ is injective,
and then $S$ is a Mycielski subset of $X$.

To prove \eqref{lem:dense-inv-scrambled-set:2}
is similar to to prove \eqref{lem:dense-inv-scrambled-set:1}. Simply, replace relations $R$ and $T$ by, respectively,
\begin{align*}
R'&=\{(x,y)\in X\times X: \limsup_{k\to\infty} d(f^k(f^m(x)),f^k(f^n(y)))\geq \delta\text{ for all } m,n\geq 0\}\\
T'&=\{(x,y)\in X\times X: \limsup_{k\to\infty} d(f^k(f^m(x)),f^k(f^n(x)))\geq \delta\text{ for all } m,n\geq 0\}. \qedhere
\end{align*}
\end{proof}

\begin{thm}\label{thm:uncountable-to-Mycielski}
Let $(X,f)$ be a dynamical system.
\begin{enumerate}
\item If $(X,f)$ has an uncountable invariant strongly scrambled set, then
it also has a Mycielski invariant strongly scrambled set.
\item If $(X,f)$ has an uncountable invariant $\delta$-scrambled set for some $\delta>0$,
then it also has a Mycielski invariant $\delta$-scrambled set.
\end{enumerate}
\end{thm}
\begin{proof}
(1)
By Theorem~\ref{thm:un-iss},
there exists a subsystem $(Y,f)$ of $(X,f)$ which has a dense uncountable invariant strongly scrambled set.
In fact we also have that $Y$ is perfect by the proof of Theorem~\ref{thm:un-iss}.
Now the result follows by Lemma~\ref{lem:dense-inv-scrambled-set}(1).

(2) The proof is similar to (1)
by using Lemma~\ref{lem:dense-inv-scrambled-set}(2) instead of Lemma~\ref{lem:dense-inv-scrambled-set}(1).
\end{proof}

\begin{proof}[Proof of Theorem~\ref{thm:B}]
The necessity follows by Remark~\ref{rem:invariant-s-s}, so we only need
to show the sufficiency. Assume that $(X,f)$ has a fixed point $p\in X$. Pick a transitive point $x\in X$.
Since $(X,f)$ is non-trivial system, $x\neq p$ and so $x$ is not a periodic point. Thus $X$ is a perfect set.
Now by Lemma~\ref{lem:dense-inv-scrambled-set}(1),
it suffices to show that $\orbp(x,f)$ is strongly scrambled.
It is clear that there exists an increasing sequence $\{k_i\}$
such that $\lim_{i\to\infty}f^{k_i}(x)=p$.
Fix $m,n\in\N_0$. By continuity of $f$ we have $\lim_{i\to\infty}f^{k_i}(f^m(x))=f^m(p)=p$ and $\lim_{i\to\infty}f^{k_i}(f^n(x))=f^n(p)=p$,
which implies that $(f^m(x),f^n(x))$ is proximal.
Similarly, we can show that $(f^m(x),f^n(x))$ is recurrent in $(X\times X,f\times f)$
as $x$ is recurrent in $(X,f)$.
Therefore $(f^m(x),f^n(x))$ is strongly scrambled.
\end{proof}

\begin{rem}
The sufficiency of~Theorem~\ref{thm:B} was also proved in~\cite{YL09} under certain conditions
in the framework of $X$ being a Polish space.
\end{rem}

Now we turn to the question when a dynamical system has an invariant $\delta$-scrambled set for some $\delta>0$.

\begin{lem}\label{lem:invariant-delta}
If a dynamical system $(X,f)$ has an invariant $\delta$-scrambled set for some $\delta>0$, then it is not uniformly rigid.
\end{lem}
\begin{proof}
Let $S\subset X$ be an invariant $\delta$-scrambled set. By the definition of $\delta$-scrambled set,
$S$~contains at least two points. Moreover, note that $S$ contains at most one periodic point. Hence we can pick up a point $z\in S$ which is not a periodic point,
and so for every $n\geq 1$, $z,f^n(z)$ are two distinct points in $S$. In particular, the pair
$(z,f^n(z))$ is $\delta$-scrambled for every $n\geq 1$ which yields that
there exists an integer $k\geq 1$ such that $d(f^k(z),f^k(f^n(z)))\geq \delta$.
If we put $x_n=f^k(z)$ then $d(f^n(x_n),x_n)\geq \delta$ which, since $n\geq 1$ was arbitrary, proves that $(X,f)$
is not uniformly rigid.
\end{proof}

\begin{proof}[Proof of Theorem~\ref{thm:C}]
The necessity follows from Remark~\ref{rem:invariant-s-s} and Lemma~\ref{lem:invariant-delta}.
Now we show the sufficiency.
As $(X,f)$ is not uniformly rigid, there is $\delta>0$ such that
for every $n\geq 1$ there exists $x_n\in X$ with $d(f^n(x_n),x_n)\geq \delta$.
Pick a transitive point $x\in X$. By Lemma~\ref{lem:dense-inv-scrambled-set}(2),
it suffices to show that $\orbp(x,f)$ is $\delta$-scrambled.
Fix $m,n\in\N_0$. Then $(f^m(x),f^n(x))$ is proximal as there exists a fixed point in $X$.
Without loss of generality, assume that $m>n$ and put $q=m-n$.
Notice that $f^n(x)$ is also a transitive point. Then there exists an increasing sequence  $\set{k_i}_{i=0}^\infty$
in $\N$
such that $\lim_{i\to\infty}f^{k_i}(f^n(x))=x_q$.
By the continuity of $f$, $\lim_{i\to\infty}f^{k_i}(f^m(x))=f^q(x_q)$.
Then $\limsup_{i\to\infty} d(f^i(f^m(x)),f^i(f^n(x)))\geq d(f^q(x_q),x_q)\geq \delta$
which proves that $(f^m(x),f^n(x))$ is $\delta$-scrambled, completing the proof.
\end{proof}

Since every uniformly rigid system must have zero topological entropy~\cite{GM89},
we have the following corollary.

\begin{cor}
If a  transitive system $(X,f)$ has positive topological entropy, then
$(X,f)$ has a dense Mycielski invariant $\delta$-scrambled set for some $\delta>0$
if and only if it has a fixed point.
\end{cor}

By Lemma~\ref{lem:connected}, if $X$ is not connected and $(X,f)$ is a transitive  system with a fixed point,
then it cannot be pointwise recurrent and thus not uniformly rigid.
This leads us to another corollary to Theorem~\ref{thm:C}.

\begin{cor}
Let $(X,f)$ be a non-trivial transitive system.
If $X$ is not connected, then $(X,f)$ has a dense Mycielski invariant $\delta$-scrambled set for some $\delta>0$
if and only if it has a fixed point.
\end{cor}

Recall that a dynamical $(X,f)$ is \emph{mildly mixing} if for every transitive  system $(Y,g)$,
the product system $(X\times Y,f\times g)$ is transitive.
Note that a non-trivial mildly mixing system cannot be uniformly rigid~\cite{GW06,HY04}.
Therefore, we easily obtain another result related to Theorem~\ref{thm:C}.

\begin{cor}
If a non-trivial dynamical system $(X,f)$ is mildly mixing,
then $(X,f)$ has a dense Mycielski invariant $\delta$-scrambled set for some $\delta>0$
if and only if it has a fixed point.
\end{cor}

The authors in~\cite{BGO10} conjecture that there exists a weakly mixing system with a fixed point
but without invariant $\delta$-scrambled sets.
By Theorem~\ref{thm:C}, it suffices to construct
a weakly mixing system with a fixed point which is also uniformly rigid.
In fact, we will construct two kinds of completely scrambled systems which are weakly mixing, proximal and uniformly rigid.
We show below the existence of such a system can be derived from results of Akin and Glasner~\cite{AG}.
As we will see the proof obtained as a combination of abstract arguments, which does not leave
too much space for modification (and also involve some other nontrivial constructions, e.g. monothetic groups with fixed point on compacta property~\cite{Gla}). In the next section we will show how construction from~\cite{KW} can be modified to construct systems whose existence is proved in Theorem~\ref{thm:A}.

Before proving Theorem~\ref{thm:A}, we need some more definitions. Given a compact metric space $(X,d)$
we denote by $C(X)$ the space of all continuous maps $f\colon X\to X$
endowed with a complete metric $\rho(f,g)=\sup_{x\in X} d(f(x),g(x))$.
Recall that a dynamical system $(X,f)$ is \emph{scattering} if
the product system $(X\times Y, f\times g)$ is transitive for every minimal system $(Y,g)$.

\begin{proof}[Proof of Theorem~\ref{thm:A}]
By Corollary~4.14(a) in~\cite{AG} there exists a non-trivial almost equicontinuous scattering system $(Y,g)$.
Since $(Y,g)$ is scattering, $(Y,g)$ is not equicontinuous and then also cannot be minimal.
Denote by $\Lambda_g$ the closure of $G=\set{g^n : n\geq 0}$ in $C(X)$.
By the definition $G$ is dense in $\Lambda_g$ and since $(Y,g)$
is uniformly rigid, there is an increasing sequence $\set{n_i}_{i=0}^\infty$ such that $\lim_{i\to \infty}\rho(g^{n_i},id)=0$.
In the language of \cite{AG}, the pair $(\Lambda_g,g)$ is a recurrent pointed monothetic group.
Note that in a complete metric space a set is totally bounded if its closure is compact (e.g. see \cite[Corollary~XIV.3.6]{Dug}).
If we denote by $\Trans_g$ the set of transitive point of $(Y,g)$,
then it coincides with points of equicontinuity for $g$ because $g$ is transitive.
By \cite[Theorem~4.2(5)]{AG} and \cite[Theorem~4.3]{AG}, we obtain that if $\Lambda_g$ is compact then so is $\Trans_g$
(in metric compatible with the topology on $Y$), which means that $Y=\Trans_g$, that is $(Y,g)$ is minimal.
This is a contradiction, and then $\Lambda_g$ is not totally bounded.

Now applying \cite[Corollary~4.19]{AG} to the recurrent pointed but not totally bounded monothetic group $(\Lambda_g,g)$,
we obtain a non-trivial weak mixing dynamical system $(Z,h)$ which extends to $(\Lambda_g,g)$,
that is there is a continuous homomorphism $\pi \colon \Lambda_g \to C(Z)$ such that $\pi(g)=h$.
But clearly, every subsystem in $Z$ extends to $(\Lambda_g,g)$ (e.g. see \cite[Lemma~4.4]{AG})
and
so by \cite[Corollary~4.14(a)(2)]{AG} all minimal subsets in $Z$ are singletons.
Since $\lim_{i\to \infty}\rho(g^{n_i},id)=0$ and $\pi$ is a continuous homomorphism,
we also have $\lim_{i\to \infty}\rho(h^{n_i},id)=0$, which means that $(Z,h)$ is uniformly rigid.
Denote by $M$ the closure of the set of fixed points in $Z$ and note that $M\neq Z$.
Let $R=M\times M \cup \Delta_Z$.  Then $R$ is a closed $h\times h$-invariant  equivalence relation in $Z$.
Then we can collapse $M$ to a single point, that is we take the quotient space
$X=Z/R$ and denote by $f$ the map induced by $h$ on $X$.
Since $(X,f)$ is a factor of $(Z,h)$, it is weakly mixing and uniformly rigid.
The fixed point $[M]_R$ is the unique minimal point in $(X,f)$, which means that $(X,f)$ is proximal.
Then $(X,f)$ is as required.
\end{proof}

\section{Weakly mixing, proximal and uniformly rigid systems: Katznelson-Weiss technique}

In their interesting paper \cite{KW}, Katznelson and Weiss outlined a method of construction
of uniformly rigid systems. This technique was extended in \cite{AAB}, where it is shown that while
these systems are always uniformly rigid, they may or may not be almost equicontinuous.
We will show how this technique can be adjusted, to obtain an uniformly rigid system
with some additional dynamical properties, such
as weak mixing etc. This will provide another proof of Theorem~\ref{thm:A}, which has better insight into the dynamics and more freedom in construction of another examples.

The above mentioned systems of Katznelson and Weiss are subsystems of the Bebutov system, that is the Hilbert cube $\mathbb{I}=[0,1]^{\N_0}$ with the metric (which generates the product topology)
$$
d(x,y) = \sum_{n=0}^{\infty}\frac{|x(n)-y(n)|}{2^n},
$$
and the left shift $\sigma$ defined by $\sigma(x)(n)=x(n+1)$.
Roughly speaking, authors of \cite{KW} define $X$ to be the closure of the orbit of a point $\alpha\in \mathbb{I}$ defined as $\alpha(n)=a_\infty(n)$,
where
$$
a_\infty\colon \R \to [0,1], \quad a_\infty(t)=\sup_{j}a_{p_j}(t),
$$
for properly chosen increasing sequences of positive integers $p_j$, and where
$a_p(t)=a_1(t/p)$ for integers $p>1$ and $a_1$
is a periodic repetition (i.e. $a_1(t+2)=a_1(t)$) of the graph of a function $a_0\colon [-1,1]\to [0,1]$ satisfying $a_0(1)=a_0(-1)$ and $|a_0(s)-a_0(t)|\leq L|s-t|$
for some $L\geq 2$ (see~\cite{AAB} for more details).

It is proved in~\cite{AAB} that the above construction leads to a uniformly rigid system
and that its unique minimal point is the constant sequence of symbol $1$ (denoted $1^\infty$).
Furthermore, if $a_0(0)\neq 1$ then $(X,\sigma)$ is not minimal (in fact, it is an infinite proximal system).
Sequences $\alpha$ constructed by the above procedure are called \emph{WK sequences} in~\cite{AAB}.
Another nice feature of $(X,\sigma)$ proved in \cite{AAB} is that if $a_0$ has strict minimum in $0$,
then $(X,\sigma)$ is almost equicontinuous.

Usually in symbolic dynamics a recurrent (or minimal) point is constructed by a consecutive concatenations of blocks. In the case of WK sequences the technique is
slightly different. Instead of concatenation, we adjust values of a function in properly chosen positions. It is a way to ensure uniform rigidity.
Unfortunately, periodicity condition in original construction puts some restrictions on the sequence.
To get weak mixing systems, we should break the periodicity condition on $a_p$, e.g. by changing functions in each step.

The following result from~
\cite[Lemma 2.1]{HYBeb} is more convenient
to verify a subsystem of the Bebutov system being uniformly rigid.
\begin{lem}\label{lem:urec}
Assume that $(Y,\sigma)$ is a transitive subsystem of $([0,1]^{\N_0},\sigma)$
and that $y\in Y$ is a transitive point.
Then $(Y,\sigma)$ is uniformly rigid if and only if there is an increasing sequence $\set{n_i}$
such that for each $\eps>0$ there is $N>0$ with $|y(n_i+j)-y(j)|<\eps$ for all $i>N$ and $j\geq 0$.
\end{lem}

Now we will proceed with a more concrete construction of a system announced in Theorem~\ref{thm:A}.
At one hand we have to ensure kind of ``periodic'' condition in Lemma~\ref{lem:urec}. On the other
hand we need to ensure the condition of Lemma~\ref{lem:WM} to get weak mixing, which means that any prefix of $\alpha$ must appear
in $\alpha$ with gaps of length $n$ and $n+1$ for some $n$. Then, instead of periodic structure of maps in original
Katznelson-Weiss construction, we need to incorporate a kind of ``shift'' by one position. Roughly speaking, we will make such shift
over ``flat" blocks in $a_\infty$ with value $1$ since on these positions our ``shift'' has no influence on value of $a_\infty$
and so does not break the condition in Lemma~\ref{lem:urec}. Of course when ``leaving'' flat regions of value $1$, we need to keep again periodicity of values.
We hope that this informal explanation will help the reader to follow our construction easier and apply it successfully in other situations.

Let $a_0\colon \mathbb{R}\to [0,1]$ be defined as
\[
a_0(t)=\begin{cases}
  0, & \text{ for } t\in [-1,1],\\
  1, & \text{ for } t\in [-3,-2]\cup [2,3],\\
  -t-1,&\text{ for } t\in[-2,-1],\\
  t-1,& \text{ for } t\in[1,2]\\
  0, & \text{ for }t\not \in [-3,3].
\end{cases}
\]

Put $p_0=3$ and for $n=1,2,\ldots$, we define $L_{n}\geq p_{n-1}^2$ and $p_n=p_0^2 L_n p_{n-1}$.
If we put $\eps_n=1/n$ then clearly $1/L_n <\eps_n$.

We put $b_0(t)=a_0(t)$ for $t\in [-p_0,p_0]$ and extend it to the whole $\R$ by
putting $b_0(t+2p_0)=b_0(t)$.
For $n=1,2,\ldots$, we define a function $c_n\colon \R \to [0,1]$ by
$c_n(t)=b_0(t/(p_{n-1}L_{n}))$.

If $a_n\colon \mathbb{R}\to[0,1]$ is defined, then we
put $b_n(t)=a_n(t)$ for $t\in [-p_n,p_n]$ and extend it to the whole $\R$ by putting $b_n(t+2p_n)=b_n(t)$.
Finally, we define $a_{n+1}\colon \R\to [0,1]$ by
\[
a_{n+1}(t)=\begin{cases}
\max\set{b_n(t),c_{n+1}(t)}, &\text{ for }t\in [-p_{n+1}, p_0 L_{n+1} p_n],\\
\max\set{b_n(t+1),c_{n+1}(t)}, &\text{ for }t\in (p_0 L_{n+1} p_n,p_{n+1}],\\
0, &\text{ for } t\not \in [-p_{n+1},p_{n+1}].
\end{cases}
\]

When we finish the above recursive procedure, we define a function $a_{\infty}\colon \R \to [0,1]$
by the formula
\[a_\infty(t)=\sup_{n\in\N_0} a_n(t),\quad \text{ for } t\in\R.\]

\begin{defn}
Using the above construction, we define $\alpha\in [0,1]^{\N_0}$ putting $\alpha(n)=a_{\infty}(n)$.
We denote $\mathbb{X} = \overline{\orbp(\alpha,\sigma)}\subset [0,1]^{\N_0}$.
\end{defn}
We are going the prove the following result.
\begin{thm}\label{thm:main_wm}
The dynamical system $(\mathbb{X},\sigma)$ is weakly mixing, proximal and uniformly rigid.
\end{thm}

\begin{rem}\label{rem_11}
Observe that for any sequence $\set{x_i}_{i=0}^\infty, \set{y_i}_{i=0}^\infty \subset [0,1]$
we have
\[
|\sup_{i} x_i - \sup_i y_i|\leq \sup_i |x_i-y_i|.\]
\end{rem}

\begin{rem}
For any $t\in \R$, $s>0$ and $n\in\N$, directly form the definition we have that
\begin{align*}
|b_0(t+s)-b_{0}(t)|&\leq s,\\
|c_n(t+s)-c_{n}(t)|&=|b_0((t+s)/(p_{n-1}L_{n}))-b_{0}(t/(p_{n-1}L_{n}))|\leq s/(p_{n-1}L_{n})\\
&\leq s\eps_n/p_{n-1}.
\end{align*}
\end{rem}

\begin{lem}\label{lem:copy_up}
For any $0\leq n\leq m$ and any $t\in [-p_n,p_n]$ we have $a_\infty(t)=a_m(t)=a_n(t)$.
\end{lem}
\begin{proof}
Fix any $t\in [-p_n,p_n]$.
We first show that $a_m(t)=a_n(t)$.
If $m=n$ it has noting to prove. Now assume that $n<m$.
Since $t\in [-p_n,p_n]$, $t/(p_nL_{n+1})\in[-1,1]$ and so $c_{n+1}(t)=0$.
Then $a_{n+1}(t)=\max \set{b_n(t),c_{n+1}(t)}=b_n(t)=a_n(t)$.
By induction, assume that we have proved that $a_{m-1}(t)=a_n(t)$.
Again by $t/(p_{m-1}L_{m})\in[-1,1]$, we have $c_{m}(t)=0$.
Then $a_{m}(t)=\max \set{b_{m-1}(t),c_{m}(t)}=b_{m-1}(t)=a_{m-1}(t)=a_{n}(t)$.

Now we show that $a_\infty(t)=a_n(t)$.
If $t\in [-p_0,p_0]$ then for every $i\geq 0$ $a_i(t)=a_0(t)$ and so the result follows.
If $t\not\in [-p_0,p_0]$ then let $0\leq k<n$ be such that
$t\in [-p_{k+1},p_{k+1}]\setminus [-p_{k},p_{k}]$. Then $a_i(t)=0$ for every $i\leq k$
and $a_i(t)=a_k(t)$ for every $i\geq {k+1}$.
Thus, $a_\infty(t)=\sup_{i} a_i(t)=a_n(t)$.
\end{proof}

\begin{lem}\label{lem:copy_extend}
For any $0\leq n \leq m$ and any $t\in\R$ we have $\vert b_m(t+2p_n)-b_m(t)\vert<\eps_n$.
\end{lem}
\begin{proof}
If $m=n$, then $b_m(t+2p_n)=b_m(t)$ since $b_m$ has a period $2p_n$.
Now assume that for some $m\geq n$ we have
$\vert b_m(t+2p_n)-b_m(t)\vert<\eps_n$ for any $t\in\R$.
We are going to show that $\vert b_{m+1}(t+2p_n)-b_{m+1}(t)\vert<\eps_n$ for any $t\in\R$.
Since $b_{m+1}$ has a period $2p_{m+1}$, we can only consider $t\in[-p_{m+1},p_{m+1}]$.

First, consider any $t\in [-p_{m+1}, p_0L_{m+1} p_m-2p_n]$. In this case, we have
\begin{align*}
|b_{m+1}(t+2p_n)-b_{m+1}(t)| &= |a_{m+1}(t+2p_n)-a_{m+1}(t)|\\
&=|\max\set{b_{m}(t+2p_n),c_{m+1}(t+2p_n)}-\max\set{b_{m}(t),c_{m+1}(t)}|\\
&\leq  \max \set{|b_{m}(t+2p_n)-b_{m}(t)|, |c_{m+1}(t+2p_n)-c_{m+1}(t)|}\\
&\leq  \max\set{\eps_n, |c_{m+1}(t+2p_n)-c_{m+1}(t)|}\\
&\leq \max\set{\eps_n, 2\eps_{m+1} p_n / p_{m+1}}\\
&\leq  \eps_n
\end{align*}
Similar calculations yield that for any $t\in [p_0L_{m+1} p_m, p_{m+1}-2p_n]$ we have
\begin{eqnarray*}
|b_{m+1}(t+2p_n)-b_{m+1}(t)| &=& |\max\set{b_{m}(t+2p_n+1),c_{m+1}(t+2p_n)}\\
&&\quad\quad\quad\quad\quad\quad-\max\set{b_{m}(t+1),c_{m+1}(t)}|\\
&\leq & \eps_n.
\end{eqnarray*}
If $t\in [p_0L_{m+1} p_m-2p_n, p_0L_{m+1} p_m]$, then
\begin{align*}
  t/(p_{m}L_{m+1})&\in [p_0-2p_n/(p_{m}L_{m+1}),p_0]\subset [2,3],\\
  (t+2p_n)/(p_{m}L_{m+1})&\in [p_0,p_0+2p_n/(p_{m}L_{m+1})]\subset [3,4].
\end{align*}
But for any $s\in [2,4]$ we have $b_0(s)=1$.
Then $b_{m+1}(t)=c_{m+1}(t)=1=c_{m+1}(t+2p_n)=b_{m+1}(t+2p_n)$.

Finally, if $t\in [p_{m+1}-2p_n,p_{m+1}]$, then
\begin{align*}
  t/(p_{m}L_{m+1})&\in [p_0^2-2p_n/(p_{m}L_{m+1}),p_0^2]\subset [8,9],\\
  (t+2p_n)/(p_{m}L_{m+1})&\in [p_0^2,p_0^2+2p_n/(p_{m}L_{m+1})]\subset [9,10].
\end{align*}
But for any $s\in [8,10]$ we have $b_0(s)=1$.
Then $b_{m+1}(t)=c_{m+1}(t)=1=c_{m+1}(t+2p_n)=b_{m+1}(t+2p_n)$.
This ends the proof.
\end{proof}

\begin{lem}\label{lem:urigid}
The dynamical system $(\mathbb{X},\sigma)$ is uniformly rigid.
\end{lem}
\begin{proof}
Fix any $\eps>0$.
Take $n>0$ such that $\eps_n<\eps$.
Fix any $t\in \R$ and let $m>n$ be such that $t\in [-p_m,p_m-2 p_n]$. Then by Lemma~\ref{lem:copy_extend} we have
\[|a_m(t+2p_n)-a_m(t)|<\eps_n\]
and by Lemma~\ref{lem:copy_up} we obtain that $a_\infty(t)=a_m(t)$ and $a_\infty(t+2p_n)=a_m(t+2p_n)$.
This proves that for every $s\in \R$ the following condition is satisfied
\[|a_\infty(s+2p_n)-a_\infty(s)|<\eps_n.\]
The proof is completed by Lemma~\ref{lem:urec}.
\end{proof}

\begin{lem}\label{lem:returns}
For any $n\geq 0$ and any $t\in [-p_{n},p_{n}]$
we have
\[a_\infty(t)=a_\infty(t-2p_0 L_{n+1}p_n)=a_\infty(t+2p_0 L_{n+1}p_n-1).\]
\end{lem}
\begin{proof}
Fix any $t\in [-p_n,p_n]$ and let $s=-2p_0 L_{n+1}p_n$. Then
$t+s\in [-p_{n+1},p_0 L_{n+1}p_n]$ and
\[\frac{t+s}{p_nL_{n+1}}=\frac{t}{p_nL_{n+1}}-2p_0\in[-7,-5].\]
But for any $s\in [-7,-5]$ we have $b_0(s)=0$.
Then $c_{n+1}(s+t)=c_{n+1}(t)=0$.
By Lemma~\ref{lem:copy_up}
\begin{align*}
a_\infty(s+t)&=a_{n+1}(s+t)=\max\set{c_{n+1}(s+t),b_n(s+t)}=b_n(s+t)\\
&=b_n(t)= a_n(t)=a_\infty(t).
\end{align*}
Similarly, if $s'=2p_0L_{n+1}p_n-1$ then $t+s'\in (p_0 L_{n+1} p_n,p_{n+1}]$ and
\begin{align*}
a_\infty(s'+t)&=a_{n+1}(s'+t)=\max\set{c_{n+1}(s'+t),b_n(s'+t+1)}=b_n(s'+t+1)\\
&=b_n(t)=a_n(t).
\end{align*}
This ends the proof.
\end{proof}

\begin{lem}\label{alpha:wm}
The dynamical system $(\mathbb{X},\sigma)$ is weakly mixing.
\end{lem}
\begin{proof}
First observe that by Lemma~\ref{lem:returns} for any $m\in \Z$ and
any $k>0$ the sequence $a_{\infty}(m-k),a_{\infty}(m-k+1),\ldots, a_{\infty}(m+k)$
is a block in $\alpha$ (in fact appears infinitely many times). This shows that $\alpha$ is recurrent.

Fix any neighborhood $U$ of $\alpha$ and fix $\eps>0$ such that if $d(\beta,\alpha)<\eps$ then $\beta\in U$.
Take an enough large integer $n$ such that $\sum_{i=p_n}^\infty 2^{-i+1}<\eps$ and put $N=2p_0 L_{n+1}p_n-1$.

Let $\beta(i)=a_{\infty}(i-N-1)$ for $i\geq 0$. Note that $\beta\in \mathbb{X}$.
By Lemma~\ref{lem:returns} we obtain that
\begin{align*}
d(\alpha,\beta)&=\sum_{i=0}^{\infty}\frac{|\alpha(i)-\beta(i)|}{2^i} = \sum_{i=0}^{\infty}\frac{|a_\infty(i)-a_\infty(i-2p_0L_{n+1}p_n)|}{2^i}\\
&\leq\sum_{i=p_{n}}^{\infty}\frac{2}{2^i}<\eps.
\end{align*}
This shows that $\beta\in U$. Note that $\sigma^{N+1}(\beta)=\alpha\in U$.
Similar calculations show that also $\sigma^{N}(\alpha)\in U$.
By Lemma~\ref{lem:WM} $(\mathbb{X},\sigma)$ is weakly mixing.
\end{proof}

\begin{lem}\label{lem:4.10}
The dynamical system $(\mathbb{X},\sigma)$ is proximal.
\end{lem}
\begin{proof}
By Lemma~\ref{lem:proximal}, it suffices to show that for any $k>0$,
the block of $k$ consecutive symbols $1$ appears in $\alpha$ syndetically.
In fact, we can show that the block of $p_n/9$ consecutive symbols $1$ appears in $\alpha$ with gap less that $2p_n$.
By Lemma \ref{lem:copy_up}, it suffices to show the following claim.
Note that when consider subblocks of $a_m$, we only consider integer coordinates.
\begin{claim}
For any $m\geq n\geq 1$,
the block of $p_n/9$ consecutive symbols $1$ appears in $a_m[-p_m,p_m]$  with gap less that $2p_n$.
\end{claim}
\noindent\textit{Proof of the Claim.}
Note that $a_n(t)\geq c_n(t)$ for $t\in[-p_n,p_n]$.
Clearly, $c_n[-p_n,p_n]$ begins
with the block of $p_n/9$ consecutive symbols $1$, so does $a_n$.
Then the block of $p_n/9$ consecutive symbols $1$ appears in $b_n$ with gap less that $2p_n$.
By the definition, we have
\begin{eqnarray*}
a_{n+1}(t)\geq b_n(t)&&\text{ for } t\in[-p_{n+1},p_0L_{n+1}p_n],\\
a_{n+1}(t)\geq b_n(t+1)&& \text{ for }t\in(p_0L_{n+1}p_n,p_{n+1}],\\
a_{n+1}(t)\geq c_{n+t}(t)&&\text{ for }t\in [p_0L_{n+1}p_n-2p_n, p_0L_{n+1}p_n+2p_n].
\end{eqnarray*}
It is easy to check that $c_{n+1}(t)=1$ for $t\in [p_0L_{n+1}p_n-2p_n, p_0L_{n+1}p_n+2p_n]$.
So the claim holds for $a_{n+1}$.
By induction, we known the claim holds for all $m\geq n$.
\end{proof}

\begin{proof}[Proof of Theorem~\ref{thm:main_wm}]
It is enough to combine Lemmas~\ref{lem:returns}, \ref{lem:urigid} and \ref{lem:4.10}.
\end{proof}

By Lemma~\ref{lem:connected}, we known that the space $\mathbb{X}$ is connected.
Unfortunately, it is hard to say what is the topological dimension of $\mathbb{X}$.
Since it is embedded in Hilbert cube, the topological dimension can be infinite.
In contrast to that, the authors of \cite{HYCS}
constructed completely scrambled homeomorphism on Cantor fan (which is dendroid), but
it is impossible to construct such a map on a dendrite (see \cite[Corollary~1.2]{IN}). So it is not completely clear which compacta admit
completely scrambled homeomorphisms and which do not
(and it is not easily related to its topological dimension).
Additionally, even if we have
a continuum of topological dimension $n$ with a transitive completely scrambled map,
then we cannot directly use the technique of \cite{HYCS} to increase
its topological dimension to $n+1$ because we will obtain a non-transitive map.

Motivated by the above discussion, we state the following open question.
\begin{que}\label{que1}
Given an integer $n\geq 1$, is it possible to construct a continuum $X$ of topological dimension $n$ and a completely scrambled
weakly mixing (or transitive) homeomorphism on it?
\end{que}

Recall that  a subset $S$ of $X$  is strongly scrambled if
for any two distinct points $x,y\in S$, the pair $(x,y)$ is proximal for $(X,f)$ and recurrent for $(X\times X, f\times f)$.

The anonymous referee of this paper suggested to call a dynamical system $(X,f)$ \emph{completely strongly scrambled} when
the whole space $X$ is a strongly scrambled set, or equivalently,
$(X,f)$ is proximal and $(X\times X, f\times f)$ is pointwise recurrent. Before we state a question about completely strongly scrambled
systems, let us make some introductory observations.

For a dynamical system $(X,f)$,
if $(X\times X, f\times f)$ is pointwise recurrent and the set of minimal points of $(X,f)$ is not dense in $X$,
then a completely strongly scrambled factor of $(X,f)$ is constructed by collapsing the closure of the set of
minimal points to a single point.
Note that if $(X,f)$ is uniformly rigid, then $(X\times X, f\times f)$ is pointwise recurrent and hence
examples constructed in Theorem~\ref{thm:A} are completely strongly scrambled systems.
By results of \cite{AAB} every transitive, almost equicontinuous and non-minimal system never
has a dense set of minimal points which is
uniformly rigid,
hence it has a completely strongly scrambled factor
(it was first observed by Huang and Ye in~\cite{HY02}).

Note that the starting point of the construction of completely scrambled systems
in~\cite{HYCS} is a completely scrambled homeomorphisms on a countable space.
But this example is not completely strongly scrambled, since it is clear that
completely strongly scrambled systems can exist only on perfect spaces.

It is shown in~\cite{AGW05} that if $X$ is a zero dimensional compact metric space
then any homeomorphism $f$ acting on $X$ is pointwise recurrent if and only if
$X$ is a union of minimal sets.
Then there is no completely strongly scrambled systems on zero dimensional spaces.
This leads to the following question.

\begin{que}
Given an integer $n\geq 1$, is it possible to construct a continuum $X$ of topological dimension $n$
and a completely strongly scrambled homeomorphism on it?
\end{que}

In \cite{HY02} the authors provided a simple method of construction of a completely scrambled system on $(n+1)$-dimensional continuum,
provided that a completely scrambled system on $n$-dimensional compact metric space is given. Unfortunately, this technique never gives
transitive system, however it produces pairwise recurrent system provided that we started form a pairwise recurrent system.

\section*{Acknowledgements}
The authors are grateful to L'ubomir Snoha for interesting discussions on properties of scrambled sets
and completely scrambled systems and to the anonymous referee of this paper for his numerous remarks and suggestions.

M. Fory\'s was supported by AGH local grant.  W.~Huang was supported by NNSF of China (11225105, 11431012).
J.~Li was supported by NNSF of China (11401362, 11471125) and NSF of Guangdong Province (S2013040014084).
P.~Oprocha was supported by the Polish Ministry of Science and Higher Education from sources
for science in the years 2013-2014, grant  no. IP2012~004272.

\end{document}